\documentclass{conm-p-l}
\def\C{\bold C}
\def\Z{\bold Z}
\def\R{\bold R}
\def\Ker{\text{\rm Ker}}
\def\Q{\bold Q}

\newtheorem{thm}{Theorem}[section]
\newtheorem*{thm*}{Theorem}
\newtheorem{lem}[thm]{Lemma}
\newtheorem{prop}[thm]{Proposition}
\newtheorem{prop1.3}[thm]{Proposition 1.3}
\newtheorem{prop1.4}[thm]{Proposition 1.4}
\newtheorem{prop1.5}[thm]{Proposition 1.5}
\newtheorem{prop1.6}[thm]{Proposition 1.6}

\newtheorem{cor}[thm]{Corollary}

\theoremstyle{definition}

\theoremstyle{remark}
\newtheorem{remark}[thm]{Remark}

\numberwithin{equation}{section}

\begin{document}

\title{Some more non-arithmetic rigid groups}


\author{Alexander Lubotzky}

\address{Institute of Mathematics, Hebrew University, Jerusalem 91904, ISRAEL}
\email{alexlub@math.huji.ac.il}
\thanks{The author is indebted to the Israel
Science Foundation and the U.S.-Israel Binational Science
Foundation for their support.} 

\date{}

\dedicatory{Dedicated to the memory of Bob Brooks}
\maketitle


\centerline{\bf 0. Introduction}

\medskip

Let $\Gamma$ be a finitely generated group. $\Gamma$ is called
{\bf rigid} if for every $n\geq 1$, $\Gamma$ admits only finitely
many isomorphism classes of irreducible representations into
$GL_n(\C)$.  Platonov (cf. \cite{PR} p. 437) conjectured that if
$\Gamma$ is a linear rigid group then $\Gamma$ is of ``arithmetic
type'' (i.e., commensurable with a product of $S$-arithmetic
groups).

In \cite{BL}, H. Bass and the author gave counter-examples to
Platonov's conjecture.  The examples $\Gamma$ there are very
special; they are subgroups of $L\times L$ when $L$ is a uniform
(arithmetic) lattice in the rank one simple Lie group
$F_4^{(-20)}$.  The proof there relies on four main ingredients:

\begin{enumerate}

\item{} The super-rigidity (\'a la Margulis) of $L$, which was proved
by Corlette \cite{C} and Gromov-Schoen \cite{GS}.

\item{} The
Ol'shanski\u\i-Rips theorem ([O]) asserting that $L$ being
hyperbolic has a finitely presented infinite quotient $H$ with no
proper finite index subgroup.

\item{}  Grothendieck's theorem \cite{GO}
which says that once the inclusion of $\Gamma$ to $L\times L$
induces an isomorphism of the pro-finite completions $\hat\Gamma$
and $\hat L\times\hat L$, the representation theory of $\Gamma$ is
the same as that of $L\times L$.

\item{} The vanishing of the second Betti number $\beta_2(L') =
\dim H^2(L',\R)=0$ for every finite index subgroup $L'$ of $L$
(proved by Kumareson and Vogan-Zuckerman (cf. \cite{VZ})).
\end{enumerate}

In this note we give a simplified version of the proof in
\cite{BL}.  We eliminate the use of ingredients (3) and (4).  Not
using (3) makes the proof more elementary, but avoiding (4) is
even more significant: our proof now works also when $L$ is any
uniform lattice in $Sp(n,1)$ (for every $n\geq 2$), groups for
which (4) does not hold (as follows from (\cite{Li}, Cor. 6.5). We
thus have many new examples which are Zariski dense in
$Sp(n,1)\times Sp(n,1)$.

Just as in \cite{BL}, the counter examples to Platonov's
conjecture constructed here are even super-rigid (see \S1).  In
summary:

\begin{thm*}  Let $G$ be either $Sp(n,1), n\geq 2$, or
$F_4^{-(20)}$ and let $L$ be a cocompact lattice in $G$.  Then
$L\times L$ contains a subgroup $\Gamma$ of infinite index which
is Zariski dense in $G\times G$, super-rigid and not of arithmetic
type.
\end{thm*}

We end the note with two suggestive remarks to be explained in \S4:

We show that if $L$ satisfies the congruence subgroup property,
then a further simplification is possible, avoiding also the use
of (1) and (2), i.e. the works of Corlette, Gromov-Schoen,
Ol'shanskii and Rips.

Furthermore, if a single uniform lattice $L$ in $Sp(n,1)$ or
$F_4^{(-20)}$ satisfies the congruence subgroup property, then
there exists an hyperbolic group \'a la Gromov which is not
residually-finite.  Moreover, there exists a hyperbolic group with
no proper finite index subgroup.  So, answering the congruence
subgroup problem in the affirmative, for one such $L$ would settle
in negative the long standing open problem on the
residual-finiteness of hyperbolic groups.

\section{FAb fibre products of rigid groups are rigid}

Throughout the paper, $\Gamma$ is a finitely generated group.
$\Gamma$ is {\bf FAb} if for every finite index subgroup $\Lambda$
of $\Gamma$, $\Lambda^{ab} := \Lambda/[\Lambda, \Lambda ]$ is
finite. $\Gamma$ is {\bf rigid} if for every $n\geq 1$, it has
only finitely many non-equivalent irreducible $n$-dimensional
complex representations.  Finally, $\Gamma$ is {\bf super-rigid}
if $A(\Gamma)^0$, the connected component of $A(\Gamma)$ is
finite-dimensional algebraic group.  Here, $A(\Gamma)$ is the
pro-algebraic completion of $\Gamma$.  There is a homomorphism
$i:\Gamma\to A(\Gamma)$ such that for every representation $\rho
:\Gamma\to GL_n(\C)$, there exists a unique algebraic
representation $\tilde\rho :A(\Gamma)\to GL_n(\C)$ with
$\tilde\rho\circ i = \rho$.

It is not difficult to prove:

\begin{prop}
$\text{super-rigid}\Rightarrow\text{rigid}\Rightarrow FAb$.
\end{prop}

\begin{prop}
If $\Lambda$ is a finite index subgroup of
$\Gamma$, then $\Lambda$ is super-rigid (resp. rigid, FAb) iff
$\Gamma$ is.
\end{prop}

\medskip

We also have:

\begin{prop}
  Assume $\Gamma$ has FAb.  Then given $n$,
there exists a finite index subgroup $\Lambda =\Lambda (n)$ of
$\Gamma$ such that for every $n$-dimensional representation
\newline $\rho :\Gamma \to GL_n(\C)$, $\overline{\rho(\Lambda)}$
is a connected group. If $\rho$ is irreducible then
$\overline{\rho (\Lambda)}$ is semisimple.
\end{prop}

\begin{proof} Denote $H=\overline{\rho (\Gamma)}$, then there exists
a finite subgroup $F$ of $H$ such that $H=H^0\cdot F$ (\cite{W}
10.10). By Jordan's Theorem (\cite{W} 9.2), $F$ has an abelian
normal subgroup of index at most $J(n)$.  Thus the same applies to
$H/H^0$.  Let now $\tilde\Lambda$ be the intersection of all the
normal subgroups of $\Gamma$ of index at most $J(n)$ and $\Lambda
= [\tilde\Lambda, \tilde\Lambda]$.  It follows that $\Lambda$ has
finite index in $\Gamma$ and for every $n$-dimensional
representation of $\Gamma$, $\rho (\Lambda)
\subseteq\overline{\rho (\Gamma)}^{\,0}$. This also implies that
$\rho (\Lambda)$ is dense in $\overline{\rho(\Gamma)}^{\,0}$ since
$\Lambda$ has finite index in $\Gamma$ and $\overline{\rho
(\Gamma)}^{\,0}$ has no finite index subgroups. Finally, if $\rho$
is irreducible, then $\overline{\rho (\Gamma)}^{\,0}$ is
reductive, but actually semisimple as it has no abelian quotient.
\end{proof}

\begin{prop}
  A finitely generated FAb group $\Gamma$ is not
rigid iff there exists a finite index subgroup $\Lambda$ of
$\Gamma$ and a simple algebraic group $H$ of adjoint type (i.e.,
$Z(H) =1$) such that $\Lambda$ has infinitely many non-conjugate
homomorphisms $\rho: \Lambda\to H$ with $\rho (H)$ Zariski dense
in $H$.
\end{prop}

\begin{proof}  Assume there are infinitely many non-equivalent
irreducible\break $n$-dimensional representations $\rho$ of
$\Gamma$ and let $\Lambda =\Lambda (n)$ as in Proposition 1.3.
Then $\overline{\rho (\Lambda)}$ is a connected semisimple group
for every such $\rho$.  There are only finitely many conjugacy
classes of semisimple connected subgroups of $GL_n(\C)$, so we may
assume that $H=\overline{\rho (\Lambda)}$ is fixed. From Clifford
theorem [W, Theorem 1.7], it follows that there are infinitely
many non-equivalent $\Lambda$-representations with Zariski-dense
image in $H$. Infinitely many of them are still non-equivalent
when $H$ is divided by its finite center.  We can therefore assume
$H$ is of adjoint type and a direct product of its simple
components. There are still infinitely many non-equivalent
homomorphisms of $\Lambda$ to one of the simple factors of $H$.
Replacing $H$ by this factor gives the result.  The other
direction is easy as $(\text{Aut}(H): \text{Inn} (H))$ is finite.
\end{proof}

We now show how new rigid (resp. super-rigid) groups can be
obtained as a fibre product of rigid (resp. super-rigid) groups.

For $i=1,2$, let $L_i$ be a finitely generated group with
epimorphisms $\rho_i$  from $L_i$ onto the same finitely presented
group $D$, and $R_i = \Ker ~ \rho_i$.  Let
$$
\Gamma = L_1 \times_{D} L_2 = \{(x,y)\in L_1\times L_2\mid
\rho_1(x) = \rho_2(y)\}
$$
be the {\bf fibre product} of $L_1$ and $L_2$ over $D$.

The projections $\pi_i$ of $\Gamma$ to $L_i$ are onto with kernels
$(1,R_2)$ and $(R_1,1)$.  Also, one can easily see that $\Gamma$
is finitely generalized  since $D$ is finitely presented.

\begin{prop}
  If $L_1$ and $L_2$ are rigid and $\Gamma$
has FAb, then $\Gamma$ is also rigid.
\end{prop}

\begin{proof}  If not, by Proposition 1.4, there exists a finite
index subgroup $\Lambda$ of $\Gamma$ and a simple algebraic group
$H$ that $\Lambda$ has infinitely many  non-equivalent
homomorphisms $\rho:\Lambda\to H$ with dense image.  Let $R'_1 =
(R_1,1)\cap\Lambda$ and $R'_2 = (1,R_2)\cap\Lambda$.  Then
$[R'_1,R'_2] = 1$ and both are normal in $\Lambda$.  Thus
$\overline{\rho (R'_1)}$ and $\overline{\rho (R'_2)}$ are
commuting normal subgroups of $H$, so one of them must be trivial.
Hence for infinitely many $\rho$'s, either $\rho (R'_1)$ is
trivial or $\rho (R'_2)$ is.   In the first case we have that
$\rho$ factors through $\pi_2$, i.e., we have a representation of
$\Lambda /R'_1 = \Lambda /(R_1,1)\cap\Lambda\simeq \Lambda
(R_1,1)/(R_1,1)$ which is commensurable to $L_2$.  Hence, this
group is not rigid and by Proposition 1.2 also $L_2$ is not rigid,
a contradiction.  The case with $\rho (R'_2) = 1$ is treated
similarly and the Proposition is proved.
\end{proof}

It is somewhat more difficult to prove, but under an additional
hypothesis, it is also true that when we start with $L_i$
super-rigid and reductive, then $\Gamma$ is super-rigid.

Recall that a group is called reductive if every representation of
it is reductive.

\begin{prop}
 Assume $L_1$ and $L_2$ are super-rigid and
reductive, $D$ has no finite index subgroup and $\Gamma$ is FAb
then $\Gamma$ is super-rigid and reductive, in fact
\newline $A(\Gamma )^0 = A(L_1)^0 \times A(L_2)^0$
\end{prop}

\begin{proof}  Let $\rho :\Gamma\to GL_n(\C)$ be a representation.
Replacing $\Gamma$ by $\Lambda$ as in Proposition 1.3, we can
assume $H = \overline{\rho (\Lambda)}$ is connected and it is
equal to its own commutator subgroup (since $\Gamma$ is FAb).

We claim that $H$ is semi-simple.  If not, it has a non-trivial
unipotent radical $U$.  Dividing by $[U,U]$ and further if needed,
we can assume that $U$ is a simple $H$-module.  So $H=U\cdot S$,
a semi-direct product when $S$ is semisimple and $U$ a simple
$S$-module. Clearly, this is not the trivial module, since $H$ has
no abelian quotient.

As in the proof of (1.5), $\overline{\rho (R'_1)}$ and
$\overline{\rho (R'_2)}$ are commuting normal subgroups.
Moreover, $\Lambda/R'_1\times R'_2\simeq D$ and $D$ has no finite
index subgroups, hence no finite dimensional representations.  It
follows that $\overline{\rho (R'_1\times R'_2)} = \overline{\rho
(R'_1)}\cdot\overline{\rho (R'_2)}=H$.

If $\overline{\rho (R'_1)}$ does not contain $U$, then
$H/\overline{\rho (R'_1)}$ has non-trivial unipotent radical and
so $\Lambda/R'_1$, which is commensurable to $L_2$ is not
reductive, a contradiction.  Hence $\overline{\rho (R'_1)}$, and
similarly $\overline{\rho (R'_2)}$, contains $U$.  But
$[\overline{\rho (R_1)},\overline{\rho (R'_2)}]=1$, since
$[R'_1,R'_2]=1$.   This implies that $[U,\overline{\rho (R'_2)}] =
[\overline{\rho (R'_1)},U]=1$ and hence $H$ acts trivially on $U$,
a contradiction.  Hence $U$ is trivial and $H$ is semisimple.
Again $\overline{\rho (R'_1)}$ and $\overline{\rho (R'_2)}$ are
normal and commuting.  This implies that (after dividing by the
center of $H$) $H=\overline{\rho (R'_1)}\times\overline{\rho
(R'_2)}$ but then $H/\overline{\rho (R'_i)}$ is an homomorphic
image of $A(L_{3-i})^0$.  This proves that $A(\Gamma)^0 =
A(L_1)^0\times A(L_2)^0$.
\end{proof}

\section{When a fibre product is FAb}

In this section we show some sufficient conditions for a fibre
product to be FAb.

Let $L$ be a finitely generated group, $\rho:L\to D$ a homomorphism onto a finitely presented group $D$ with kernel $R$, and $\Gamma = L\times_{D} L$ the fibre product of $L$ over $D$,
$$
\Gamma = \{(x,y)\in L\times L\mid \rho (x) = \rho (y)\}.
$$
Note that $\Gamma = (R,1)\Delta (L) = (1,R)\Delta (L)$ when
$\Delta (L)$ is the diagonal embedding of $L$ into $\Gamma$,
$\Delta (L) = \{(x,x)\mid x\in L\}$.

\begin{lem}
  If $L$ is FAb and if for every finite index normal
subgroup $L_0$ of $L$, and every finite index subgroup $R_0$ of
$R$ which is normal in $L$, $[L_0,R_0]$ is of finite index in
$R_0$, then $\Gamma$ is FAb.
\end{lem}

\begin{proof}  Let $\Gamma_0$ be a finite index normal subgroup of
$\Gamma$, $L_0 = \Gamma_0\cap\Delta (L)$ and $R_0 =
(R,1)\cap\Gamma_0$.  By abuse of notation we will consider $R_0$
as a subgroup of $L$ and at the same time a subgroup of $(L, 1)$.
>From our assumption, it follows that $[R_0,L_0]$ is of finite
index in $R_0$ and hence in $(R,1)$. As $L$ is FAb, $[L_0,L_0]$ is
of finite index in $\Delta (L)$. This implies that $[R_0L_0,
R_0L_0]$ is of finite index in $\Gamma$. As $\Gamma_0\supseteq
R_0L_0$, we get that $[\Gamma_0,\Gamma_0]$ is of finite index in
$\Gamma$.
\end{proof}

Let now $L$ be a finitely generated group with FAb.  Assume $L$
has an infinite finitely  presented quotient $D$ with kernel $R$,
such that $\hat D = \{ 1\}$ when $\hat D$ is the profinite
completion of $D$, i.e., $D$ has no proper finite index subgroup.
Clearly $D^{ab} = \{ 1\}$, i.e, $D$ is a perfect group.  It has
therefore a universal central extension \begin{equation}\tag{$*$}
1\rightarrow H_2(D)\rightarrow \bar D\rightarrow D\rightarrow 1
\end{equation}
(cf. \cite{M} \S 5).

\begin{lem}
 Let
\begin{equation}\tag{$**$}
1\rightarrow C\rightarrow E\rightarrow D\rightarrow 1
\end{equation}
be a central extension of $D$ such that $E^{ab}$ is finite.  Then
$rk(C)\leq rk(H_2(D))$.  (For an abelian group $A$ we denote
$rk(A) = \dim_{\Q} (A\otimes_{\Z}\Q)$).
\end{lem}

\begin{proof}  There is a homomorphism $\psi$ from the universal
central extension $(*)$ to $(**)$.  We claim that $\psi (\bar D)$
is of finite index in $E$.  Indeed, $C\psi (\bar D) =E$ and $C$ is
central. Hence $\psi (\bar D)\vartriangleleft E$ and the quotient
is abelian.  But $E^{ab}$ is finite, hence $(E:\psi (\bar D)) <
\infty$.  This implies that $\psi (H_2(D))$ is of finite index in
$C$ and hence $rk(C)\leq rk(H_2(D))$.
\end{proof}

\begin{lem}
 For a finite index normal subgroup $L_0$ of $L$ denote $R_0 =
R\cap L_0$.  Then $rk(R_0/[L_0,R_0])\leq rk(H_2(D))$ for every
$L_0\vartriangleleft L$ of finite index.
\end{lem}

\begin{proof}  The map $L_0\to D$ is onto since $\hat D=1$ and so:
$$
1\to R_0/[L_0,R_0] \rightarrow L_0/[L_0,R_0] \rightarrow D\rightarrow 1
$$
is a central extension of $D$.  As $L$ is FAb, $L^{ab}_0$ is
finite and our Lemma follows from Lemma 2.2.
\end{proof}

Given $L$ and $\rho :L\to D$ with $R = \Ker ~\rho$ as before, choose now $L_0$ of finite index with $R_0 = R\cap L_0$ and $rk(R_0/[L_0,R_0])$ maximal among all possible such $L_0$.  As $\hat D = \{ 1\}$, it follows that $L_0/R_0\simeq D$ and we have a central extension
\begin{equation}\tag{2.1}
1\rightarrow R_0/[L_0,R_0] \rightarrow
D_0:=L_0/[L_0,R_0]\rightarrow D=L_0/R_0\rightarrow 1
\end{equation}
 Now, $R_0/[L_0,R_0]$ is dense in the profinite
topology of $D_0$ since $\hat D = \{ 1\}$, hence $\hat D_0$ is
abelian.  But, $L_0$ is FAb, so $\hat D_0$ is finite.  Replace
$L_0$ by a finite index subgroup $L'_0$ so that
$D_1=L'_0/([L_0,R_0]\cap L'_0)$ satisfies $\hat D_1 = \{ 1\}$ and
in particular $D_1^{ab} = \{ 1\}$.  Let us now rename and call
$L:= L'_0$, $\tilde D:=D_1$ and $\tilde R = \Ker (L\to \tilde D)$.
So $\tilde D$ is a central (possibly infinite!) extension of the
original $D$ and $\Hat{\Tilde D} = \{ 1\}$.  We get the exact
sequence:
\begin{equation}\tag{2.2}
1\rightarrow\tilde R\rightarrow L\rightarrow\tilde D\rightarrow 1.
\end{equation}

The crucial point is:

\medskip

{\bf Claim:} \  For every normal subgroup $L_1$ of finite index in
$L$ and every $R_1$ of finite index in $\tilde R$ which is normal
in $L$, $[L_1,R_1]$ is of finite index in $R_1$.

\begin{proof}  The group $\tilde R/[L_1,R_1]$ is a finitely
generated virtually abelian group.  Hence $L/[L_1,R_1]$ is an
extension of the form $1\to C\to L/[L_1,R_1]\to D\to 1$ where $C$
is virtually abelian and $\hat D = \{ 1\}$.  Moreover, $L$ is FAb.
So, all this implies that $\widehat{(L/[L_1,R_1])}$ is finite.
Hence, whenever $L/[L_1,R_1]$ is mapped into a residually-finite
group, its image is finite.  This applies, in particular, to the
image of $L/[L_1,R_1]$ in its action (by conjugation) on the
finitely generated virtually abelian group $\tilde R/[L_1,R_1]$.
Thus for some finite index subgroup $L_2$ of $L_1$, $[L_2,\tilde
R]\subseteq [L_1,R_1]$. By the maximality choice of $L_0$, it
follows that $[L_2,\tilde R]$ is of finite index in $\tilde R$ and
hence $[L_1,R_1]$ is also of finite index there.
\end{proof}

To summarize, by replacing $D$ by $\tilde D$ (and the original $L$
by a subgroup of finite index), we get an exact sequence:
\begin{equation}\tag{2.3}
1\rightarrow\tilde R\rightarrow L\rightarrow\tilde D\rightarrow 1
\end{equation}
which satisfies the claim.  Hence by Lemma 2.1, we have:

\begin{cor}
  The fibre product $\Gamma = L\times_{\tilde
D} L$ is FAb.
\end{cor}

\section{The main result}

In this section, we pick the fruits of the preparations in the
previous two sections and construct non-arithmetic super-rigid
groups.

Propositions 1.5 and 1.6 show how to get new (super) rigid groups
from old ones.  The standard examples of rigid groups are
irreducible lattices in semisimple groups of higher rank, but
these do not have infinite quotients, so fibre products of them
are commensurable to other arithmetic groups.  This leads us to
lattices in $Sp(n,1)$ and $F_4^{(-20)}$.  They are super-rigid
(\cite{C},\cite{GS}) and at the same time have many infinite
quotients.

So, from now on in this section, let $L$ be a torsion-free uniform
(=cocompact) lattice in one of the groups $Sp(n,1)$ or
$F_4^{(-20)}$.  It is a hyperbolic group and hence by a result
proved independently by Ol'shanskii and Rips (see \cite{O}), it
has a finitely presented infinite quotient $\rho :L\to D$ where
$D$ has no proper finite index subgroups.  Replace $D$ now by
$\tilde D$ (and $L$ by a finite index subgroup, also  called $L$)
as in \S2.  Let $\tilde\rho :L\to\tilde D$ be the new map and
$\Gamma = L\times_{\tilde D} L$ the fibre product.  (The reader
may note that at this point our method differs from the one in
\cite{BL}, where $L\times_{D} L$ was used).

\begin{thm}
 $\Gamma$ is a super-rigid group. In fact,
$A(\Gamma)^0 = A(L)^0\times A(L)^0$.
\end{thm}

\begin{proof}  By Corollary 2.4, $\Gamma$ is FAb.  Hence Proposition
1.6 can be applied to deduce that $\Gamma$ is super-rigid (and
hence rigid) and $A(\Gamma)^0 = A(L)^0\times A(L)$.
\end{proof}

The proof actually shows that every representation of $\Gamma$ can
be extended, on a finite index subgroup, to a representation of
$L\times L$.  One can now repeat the standard argument given in
(\cite{BL}, pp.~1171--1172) to show that $\Gamma$ is not
(virtually) isomorphic to any lattice in a product of linear
algebraic groups (over archimedean or non-archimedean fields).  So
$\Gamma$ is not ``of arithmetic type'' giving the desired
counter-example to Platonov's conjecture. \vskip .15in

\section{Remarks on the congruence subgroup property}

Let $L$ be a uniform lattice in $Sp(n,1)$ or $F_4^{(-20)}$ as in
section 3.  Such an $L$ is an arithmetic lattice (\cite{GS}).  The
question whether $L$ satisfies the congruence subgroup property
(CSP, for short) is still open.  Serre's conjecture, posed in
\cite{S}, suggests that lattices in rank one groups do not have
CSP, while lattices in higher rank simple groups do.  So by this
conjecture, $L$ is not expected to have the CSP.  But this
conjecture was made in 1970.  Since then, it has been shown that
in spite of $Sp(n,1)$ and $F_4^{(-20)}$ being rank one groups,
lattices in them behave in many ways (property $T$,
super-rigidity, arithmeticity, etc.) like higher rank lattices.
One may, therefore, suggest that they do have the CSP.  As of now,
the answer is not known for any single such $L$.

We make here two short remarks showing that an affirmative answer
for the CSP for $L$ would have two interesting corollaries:

\begin{remark}
 If $L$ has the CSP, then the above construction
(in \S3) of $D$ (and hence of $\Gamma$) can be done without
appealing to the work of Ol'shanskii and Rips.  In fact, if $N$ is
any infinite normal subgroup of $L$ (which is finitely generated
as normal subgroup) of infinite index, then $D=L/N$ has only
finitely many finite index subgroups.  Indeed, if not, it has
infinitely many finite index normal subgroups.  Pulling them back
to $L$, we get infinitely many normal {\bf congruence} subgroups
of $L$ all containing $N$.  But it is not difficult to prove that
for $L$ being an arithmetic subgroup of a simple algebraic group,
the intersection of any infinite collection of normal congruence
subgroups must be finite and central.  This shows that $D=L/N$ has
only finitely many finite index subgroups.  We can replace $D$ by
their intersection (and $L$ by the preimage) to get the desired
finitely presented quotient without any finite index subgroup.
\end{remark}

It is of interest to recall that the CSP implies super-rigidity.
Hence, if $L$ has CSP, we can produce a counter-example without
appealing to the work of \cite{C} and \cite{GS}.  So the proof
would not use any of the four ingredients (1)-(4) listed in the
introduction.

\begin{remark} 
  If there is one such uniform lattice $L$
satisfying CSP, then there exists a hyperbolic group without any
finite index subgroup and in particular, a non-residually-finite
hyperbolic group.
\end{remark}

Indeed, $L$ is hyperbolic (but residually finite).  For ``many"
$g\in L$, the normal closure $N$ of $g$ in $L$ is an infinite
subgroup of infinite index and $D=L/N$ is also hyperbolic (see
\cite{D}). Repeating the argument from remark 4.1, we deduce that
$D$ has a finite index subgroup (hence also hyperbolic) with no
finite index subgroup.  Thus answering the congruence subgroup
problem in the affirmative for one uniform lattice in $Sp(n,1)$ or
$F_4^{(-20)}$, would solve the long standing problem of the
existence of a non-residually finite hyperbolic group.

\bibliographystyle{amsalpha}

\end{document}